\documentclass[12pt]{amsart}
\usepackage[margin=2.5cm]{geometry}
\usepackage{amsmath}
\usepackage{amsfonts}
\usepackage{amssymb}
\usepackage{xcolor}
\renewenvironment{proof}{{\bfseries Proof.}}{\qed}
\numberwithin{equation}{section} 
\newtheorem{theorem}{Theorem}[section] 
\newtheorem{pro}[theorem]{Proposition} 
\newtheorem{cor}[theorem]{Corollary} 
\newtheorem{lemma}[theorem]{Lemma} 
\newtheorem{thm}[theorem]{Theorem}
\theoremstyle{definition}
\newtheorem{defn}[theorem]{Definition} 
 
\newcommand{\gt}{\tilde{\gamma}}

\newcommand{\n}{\newline}

\newcommand{\tb}{\textbf}

\newcommand{\imp}{\Rightarrow}
\newcommand{\z}{\mathbb{Z}}

\newcommand{\g}{\gamma_{r}}

\newcommand{\io}{\iota }
\newcommand{\secref}[1]{Section~\ref{#1}}
\newcommand{\thmref}[1]{Theorem~\ref{#1}}
\newcommand{\lemref}[1]{Lemma~\ref{#1}}

\newcommand{\corref}[1]{Corollary~\ref{#1}}
\newcommand{\eqnref}[1]{~{\textrm(\ref{#1})}}

\newcommand{\floor}[1]{\lfloor #1 \rfloor}

\numberwithin{equation}{section}

\usepackage[backref]{hyperref}

\begin{document}
	\title[Spherical growth of reciprocal classes in the Hecke Groups]{ Spherical growth of reciprocal classes in the Hecke Groups }
	\author{Debattam Das  \and Krishnendu Gongopadhyay}
	\address{Indian Institute of Science Education and Research (IISER) Mohali,
		Knowledge City,  Sector 81, S.A.S. Nagar 140306, Punjab, India}
	\email{debattam123@gmail.com}
    	\address{Indian Institute of Science Education and Research (IISER) Mohali,
		Knowledge City,  Sector 81, S.A.S. Nagar 140306, Punjab, India}
	\email{ krishnendu@iisermohali.ac.in}
	\subjclass[2020]{Primary 20H10; Secondary 11F06, 05E16, 37C35}
	\keywords{ Hecke groups, reciprocal elements, reciprocal geodesics, growth rate}
	
	\iffalse
\begin{abstract}
Let $r \in \mathbb{Z}$, $r>2$,  and let $\Gamma_{2r} \cong (\mathbb{Z}/2\mathbb{Z}) \ast  (\mathbb{Z}/2r\mathbb{Z})$ be the Hecke group. Let $\mathcal{N}_\ell$ denote the set of conjugacy classes of reciprocal elements with cyclically reduced word length exactly $\ell$ in the standard generating set of $\Gamma_{2r}$. 

Let $P(x)$ be the polynomial
\[
P(x) = x^{r+1} - 2\sum_{i=1}^{r-1} x^{i} - 1.
\] 
Let $\rho$ be the unique positive root of $P$ and $C$ be a constant depending on $r.$
We prove that, as $\ell \to \infty$,
\[|\mathcal{N}_\ell| \sim C\rho^{\left\lfloor \tfrac{\ell+1}{2} \right\rfloor}.
\]

Our approach exploits the free product structure of $\Gamma_{2r}$, together with a combinatorial counting function and recurrence relations arising from cyclically reduced representatives. We further show that the asymptotic growth rate of primitive reciprocal conjugacy classes of word length $\ell$ coincides with that of $\mathcal{N}_\ell$. 

This result extends previous work for odd values of $p$ and provides a uniform asymptotic bound for all Hecke groups.
\end{abstract}
	\fi

\begin{abstract}
Let \(p>2\) be an integer and let \(\Gamma_{p}\cong \mathbb{Z}/2\mathbb{Z}\ast\mathbb{Z}/p\mathbb{Z}\) be the Hecke group. In this paper, we restrict to even Hecke groups, i.e., \(p=2r\) for \(r>1\). For each \(\ell\), let \(\mathcal N_\ell\) denote the set of conjugacy classes of reciprocal elements of word length \(\ell\). We determine the spherical growth of these classes.

Let
\[
P(x)=x^{r+1}-2\sum_{j=1}^{r-1}x^{\,r-j}-1,
\]
and let \(\rho\in(\sqrt{2},2)\) be its unique positive real root. Then there exists a constant \(C_1>0\) such that, as \(\ell\to\infty\),
\(
|\mathcal N_{2\ell}|\sim C_1\rho^\ell,
\)
and
\(
|\mathcal N_{2\ell-1}|\sim C_2\rho^\ell,
\)
where \(C_2=\rho^{-u-1}C_1\). Here, reciprocal classes of odd length occur only if \(r\) is even, i.e., \(r=2u\). The same asymptotics also hold for primitive reciprocal conjugacy classes.
\end{abstract}

    \date{\today}
	
	\maketitle	
	
	\section{Introduction}\label{1}

    A non-trivial element $g$ in a group $\Gamma$ is called \emph{reversible} if there exists $h \in \Gamma$ such that $hgh^{-1} = g^{-1}$.   {An element $h\in \Gamma$ conjugating $g$ to $g^{-1}$ is called a \emph{reciprocator} of $g$.}
If the reciprocator $h$ is an involution (i.e., $h^2 = 1$), then $g$ is said to be \emph{strongly reversible}. Every strongly reversible element, except involutions themselves, can be written as a product of two involutions. For various perspectives on reversibility, see \cite{BR}, \cite{LR}, \cite{Sa}, \cite{OS}.
In Fuchsian groups, a torsion-free element is reversible if and only if it is strongly reversible; see \cite[Theorem 1.1]{DG}. Moreover, if $hgh^{-1} = g^{-1}$ for some $g \in \Gamma$, then $h$ must be an involution. In the modular group $\mathrm{PSL}(2, \mathbb{Z})$, such elements are referred to as \emph{reciprocal} (following Sarnak \cite{Sa}), and we adopt this terminology. A conjugacy class that contains a reciprocal element is called a \emph{reciprocal class}.
If $g$ is a hyperbolic reciprocal element in  {a Fuchsian group $\Gamma$}, then the  {translation} axis of $g$ projects to a closed geodesic on $\mathbb{H}^2/\Gamma$. A reciprocal element implies that this  {closed} geodesic passes through an even-order cone point; such geodesics are called \emph{reciprocal geodesics}. Thus,  {hyperbolic }reciprocal classes bijectively correspond to the reciprocal geodesics on $\mathbb{H}^2/\Gamma$.

 Reciprocators of a reciprocal element are unique up to multiplication by elements of the centralizer \(Z_\Gamma(g)\); see \cite{BR}. If a reciprocal element \(g \in \Gamma\) is hyperbolic and admits a reciprocator \(h \in \Gamma\), then \(\langle g,h\rangle\) is an infinite dihedral subgroup of \(\Gamma\). Thus, classifying reciprocal geodesics reduces to classifying such dihedral subgroups up to conjugacy. In particular, each equivalence class of reciprocal geodesics corresponds to the conjugacy class of a reciprocal element, equivalently to the conjugacy class of the infinite dihedral subgroup generated by a hyperbolic reciprocal element \(g\) and its reciprocator \(h\).
 
In recent years, the asymptotic growth and distribution of reciprocal geodesics have been studied in a variety of settings,  in the modular group \cite{Sa}, \cite{BK},  on hyperbolic surfaces \cite{ES}, and more generally on negatively curved orbifolds \cite{PP}.     Sarnak investigated reciprocal classes in the modular group $\mathrm{PSL}(2, \mathbb{Z})$, obtaining asymptotic growth and equidistribution results; see \cite{Sa}. Bourgain and Kontorovich further studied the growth rate of \emph{low-lying} reciprocal geodesics on the modular surface \cite{BK}. Erlandsson and Souto extended these results to arbitrary  {orientable} two-dimensional hyperbolic orbifolds; see \cite{ES}. 
Parkkonen and Paulin \cite{PP} have recently generalized the results of Sarnak and Erlandsson--Souto to strongly reversible periodic orbits of the geodesic flow on pinched negatively curved Riemannian orbifolds.  In all these works, the hyperbolic length of geodesics serves as the measure for asymptotic growth.

Another approach, based on the free product structure of the modular group, was used by Basmajian and Suzzi Valli. They derived asymptotic growth rates for reciprocal classes using the lengths of the associated combinatorial words; see \cite{BV}.     Following this approach, Marmolejo, in \cite{Ma},  obtained the spherical growth for reciprocal classes in a Hecke group $\Gamma_p$ for $p$ odd. Here, by {\em spherical growth}, we mean the asymptotic count of the number of distinct reciprocal classes represented by words of exactly a given length. 

 {Recall that the Hecke group $\Gamma_p$  is generated by the M\"obius transformations acting on $\mathbb{H}^2$,
\[
\iota: z \mapsto -\frac{1}{z}
\quad \text{and} \quad
\alpha_p: z \mapsto z + \lambda_p, 
\] ~ $p \geq 3$, \hbox{ where}~ $\lambda_{p}=2\cos\pi/p$. Geometrically we can see it as a free product of $\mathbb{Z}/2\mathbb{Z} $ and $\mathbb{Z}/p\mathbb{Z}$, where $\mathbb{Z}/2\mathbb{Z}$ is generated by $\io$ and $\mathbb{Z}/p\mathbb{Z}$ is generated by $\gamma_p=\io\alpha_p$.
	In particular, $\Gamma_3$ is the modular group. Note that the map $ \gamma_p :z \mapsto -\dfrac{1}{z+\lambda_{p}} $ has order  $ p $. We denote $ \gamma_p=\iota \alpha_p$. We will fix this presentation for the Hecke group, i.e., $\langle~\io,\gamma_p~|~\io^2=\gamma_p^p=1 ~\rangle$ throughout the paper. If $p=2r$ with $r\in \mathbb{N}-\{1\}$, $\gamma_p^r$ is an involution. We denote $\tilde{\gamma}=\gamma_p^r$. }

    As follows from \cite{DG}, the combinatorics of the reciprocal words are relatively simpler for $p$ odd. This is due to the presence of a unique involution up to conjugacy that conjugates a reciprocal element $g$ to $g^{-1}$.  {For even $p$}, this is not the case as there are two choices for such an involution, up to conjugacy. Accordingly, it requires further analysis to resolve the complication when $p$ is even. In this paper, we resolve this complication and obtain the corresponding spherical growth for reciprocal classes in the even case. 
	
  {The spherical growth of the reciprocal classes for odd $p$ follows from the work of Marmolejo~\cite{Ma}.}
 {\begin{thm}\cite[Theorem 5.4.1]{Ma}

 Let $P(z) = z^{m+1} - 2z^{m-1} - \cdots - 2z - 2$ where $m = \frac{k-1}{2}$. 
The conjugacy classes of primitive reciprocal words of length $2t$ in $\mathbb{Z}_2 * \mathbb{Z}_k$ for $k$ odd, are asymptotic to $c_k \rho_k^t$ where $\rho_k$ is the unique positive root of $P(z)$ in the interval $[\sqrt{2},2]$. 
Moreover, the $\rho_k$ are strictly increasing and $\rho_k \to 2$ as $k \to \infty$.
\end{thm}}

 We  {extend} Marmolejo's result for a Hecke group $\Gamma_p$, with even $p$. In particular, we obtain an asymptotic estimate for the size of the reciprocal classes in $\Gamma_p$ when $p$ is even. 
 
  {Thus, throughout the paper we assume that $p = 2r$ is even, and we write $\g$ in place of $\gamma_p$.}   {Let $\mathcal{N}_\ell$ denote the set of conjugacy classes of reciprocal elements of word length $\ell$ in $\Gamma_p$. Thus, the primary objects of interest are conjugacy classes, not individual words.} 
 
  {Every conjugacy class in $\Gamma_p$ contains a cyclically reduced representative, and the length of the class is defined as the length of any such representative (a formal definition will be provided later). Moreover, any two cyclically reduced representatives of the same conjugacy class differ by a cyclic permutation of their letters.} 

 We prove the following.
	
	\begin{theorem}\label{mthm}
Let
\[
P(x)=x^{r+1}-2\sum_{j=1}^{r-1}x^{\,r-j}-1,
\]
and let $\rho\in(\sqrt{2},2)$ be its unique positive real root. Let \(\mathcal N_\ell\) denote the set of conjugacy classes of reciprocal elements in the Hecke group \(\Gamma_{2r}\) of word length \(\ell\). Then there exists a constant \(C_1>0\) such that, as \(\ell\to\infty\),
\[
    |\mathcal N_{2\ell}| \sim C_1 \rho^\ell,
    \qquad
    |\mathcal N_{2\ell-1}| \sim C_2 \rho^\ell,
    \]
    where $C_2=\rho^{-u-1}C_1$. Here, reciprocal classes of odd length occur only if \(r\) is even, i.e., $r=2u$.
\end{theorem}
    \medskip This yields the following corollary.
   \begin{cor}\label{cor1}
Let $\mathcal N_\ell^{p}$ denote the set of conjugacy classes of primitive reciprocal elements in the Hecke group $\Gamma_{2r}$ of word length $\ell$. Then the primitive reciprocal conjugacy classes satisfy the same asymptotic formulas as in Theorem~\ref{mthm}.

\begin{enumerate}
    \item If $r$ is odd, then
    \[
    |\mathcal N_{2\ell}^{p}| \sim C_1 \rho^\ell \qquad (\ell\to\infty),
    \]
    and there are no primitive reciprocal conjugacy classes of odd word length.

    \item If $r$ is even, then
    \[
    |\mathcal N_{2\ell}^{p}| \sim C_1 \rho^\ell,
    \qquad
    |\mathcal N_{2\ell-1}^{p}| \sim C_2 \rho^\ell
    \qquad (\ell\to\infty).
    \]
\end{enumerate}
\end{cor}
        In addition to the above asymptotic count, we provide explicit combinatorial formulas for the number of conjugacy classes in each reciprocal type, see \thmref{count1}. We have used Marmolejo's counting function $\Psi^r_n(x)$, cf.~\cite{Ma}, which counts the number of integer solutions to the equation $\sum_{i=1}^n p_i = x$, with $0 < p_i \leq r$. This extends Marmolejo's count to reciprocal classes in $\Gamma_{2r}$ for all values of $r>1$.

	We prove the results using the combinatorial structure of $\Gamma_{2r} \approx \mathbb{ Z}/2\mathbb{Z} \ast \mathbb Z/2r\mathbb{Z}$ to determine the spherical growth.   
	Each reciprocal class contains some cyclically reduced words, which can be expressed in terms of the generators $\iota$ and $\g$. Then a reciprocal element in $\Gamma_{2r}$ is  {conjugated} to one of the following forms: $\iota K \io K^{-1}$, $\gt K \gt K^{-1}$, or $\iota K \gt K^{-1}$, for some word $K\in  \Gamma_{2r}$. We refer to a word of the form $\iota \g ^{k_1}\io \g^{k_2}\dots \io \g^{k_n}$, that is, starting with $\iota$ and ending with $\g$, as an \emph{$\iota-\g$ word}. Using these conjugacy representatives, we derive a summation formula for their word lengths, which in turn yields a recurrence relation associated with each reciprocal class. Finding a closed-form solution for this recurrence relation provides the desired estimate.

	We briefly outline the structure of this paper. After recalling some preliminary notions, we note preliminary observations about the counting function in \secref{prel}. We estimate counts for the reciprocal classes in \secref{3}. Finally, we prove \thmref{mthm} and \corref{cor1} in \secref{4}.

	\medskip
	\subsection{Notations:}\label{not} We have used the following notations throughout the paper.
	\begin{enumerate}
		\item For some real number $x,$ $\lfloor x\rfloor$ is largest integer less than or equal to $x$ and $\lceil x\rceil$ is the smallest integer greater than equal to $x$. 
	\item Let $f,g : D \to \mathbb{R}_{>0}$ with $D \subseteq \mathbb{R}$. We write
\[
f(x)=O(g(x)) \quad \text{as } x \to \infty
\]
if there exist constants $M>0$ and $x_0 \in D$ such that
\[
|f(x)| \le M\,|g(x)| \quad \text{for all } x \ge x_0.
\]

We write $f \sim g$ if their ratio $\frac{f(x)}{g(x)}\to 1$ as $x\to \infty$.

In particular, $f \sim g$ implies $f(x)=O(g(x))$ as $x \to \infty$, but the converse need not hold. \end{enumerate}
	\section{Preliminaries}\label{prel}
	%In \secref{3}, we will see the absolute number of conjugacy classes of different type of reciprocal elements
	
	\subsection{Reciprocal  {elements} in Hecke groups}
     {In the case of Hecke groups, elements of order two (i.e., involutions) are also reciprocal; however, since they are finite up to conjugacy, we restrict our discussion to infinite reciprocal elements. }In \cite{DG}, several types of reciprocal elements for Hecke groups were introduced; we recall them below.
	\begin{defn}
		The reciprocal  {element} is said to be \textit{symmetric} if it can be seen as a product of $\io$ and conjugates of $ \io $.
	\end{defn}
	\begin{defn}
		The reciprocal element is said to be \textit{$ p $-reciprocal} if it can be seen as a product of $\gt$ and conjugates of $ \gt $.
	\end{defn}
	\begin{defn}
		The reciprocal element is said to be \textit{symmetric $ p $-reciprocal} if it can be seen either as a product of $\io$ and conjugates of $ \gt $ or as a product of $\gt$ and conjugates of $ \io $.
	\end{defn}

    \subsection{Words}

In this subsection, we define word length in free products and in conjugacy classes, cf. \cite{MKD} for details. 

Let \( G=\langle S \mid R \rangle \) be a group with finite generating set \(S\) and relators \(R\).  A \textit{word} \(W\) in \(G\) is a finite sequence of generators and their inverses:
\[
W = s_1^{\epsilon_1}s_2^{\epsilon_2}\cdots s_n^{\epsilon_n}, \quad \text{where } \epsilon_i=\pm1 \text{ and } s_i\in S.
\]

A \textit{subword} is a consecutive subsequence of a word. A word is \textit{reduced} if it contains no subword of the form \(s_i^{\pm1}s_i^{\mp1}\).

A reduced word is \textit{cyclically reduced} if no cyclic permutation of its letters produces a non-reduced word. For example, in \(G=\langle a,b\rangle\), the word \(ab^{2}a^{3}ba^{-1}\) is not cyclically reduced; its cyclically reduced form is \(b^{2}a^{3}b\).

We can also define word in free products of groups. Let $G_1$ and $G_2$ are two groups with the presentations $\langle S_1|R_1\rangle$ and $\langle S_2|R_2\rangle$ respectively. Then the free product of $G_1$ and $G_2$ is $\tilde{G}=G_1\ast G_2=\langle S_1~\sqcup~S_2~|~ R_1~\sqcup~R_2\rangle$ and $G_1, G_2$ are the free factors of $\tilde{G}.$ 
An element $g=a_1a_2\dots a_n$ for all $a_i\in S_1\sqcup S_2 $, is a \textit{reduced word} in $\tilde{G}$ if any of $a_i$'s not $Id$ and any two consecutive $a_i$'s do not belong to the same $G_j$ for $j\in \{1,2\}.$ For example, if $\rm{ PSL(2,\mathbb{Z})}$ $=\mathbb{Z}/2\mathbb{Z}\ast \mathbb{Z}/3\mathbb{Z}=\langle~ x,y~|~ x^2, y^3~\rangle$, then $Id, x$ are the
elements of $\mathbb{Z}/2\mathbb{Z}$ and $Id, y, y^2$ are the elements of $\rm PSL(2,\mathbb{Z)}$ the sequence $xyxy^2$ is reduced ; the sequence $yy^2$ is not reduced since the consecutive terms $y$ and
$y^2$ are in the same free factor.

Similarly we can also define \textit{cyclically reduced word}. The element $g$ in $\tilde{G}$ is cyclically reduced if any cyclic permutation of $a_i$'s produce a reduced word.  
\begin{thm}\cite[Theorem 4.1]{MKD} 
    Each element $g$ of $G_1 * G_2$ can be uniquely expressed as a
product
 $g = g_1g_2\dots g_t
$,  where $g_1 g_2 . . .  g_t$ is a reduced sequence.

\end{thm}
From this theorem we can define the length of a word in free products. 
    \begin{defn}\label{word}
{The \textit{length} of a word in a free product of two cyclic groups is the number of generators in its reduced form.}
\end{defn}

 {For the Hecke group $\Gamma_{2r}=\langle \io,\g \mid \io^{2},\, \g^{2r}\rangle$, the length of the reduced word $\io\g^{k_{1}}\io\g^{k_{2}}\dots \io\g^{k_{n}}\io$ is 
\[
1+\sum_{i=1}^{n}\bigl(1+|k_i|\bigr),
\]
where $-r< k_i \leq r$ for all $i\in\{1,2,\dots,n\}$.}

 {The length is defined only for reduced words. If $\omega$ is chosen to be cyclically reduced in a group $G$, then its length is minimal among all words in its conjugacy class. Hence we obtain the following definition of the length of a conjugacy class.}

    \begin{defn}
		 {The \emph{length} of a conjugacy class in $\Gamma_{p}$ is the minimum length of a cyclically reduced word in that conjugacy class.} 
		
		If a conjugacy class contains a cyclically reduced word with minimal length $l$, we loosely call it a \emph{conjugacy class of length $l$}. 
	\end{defn}

	\subsection{A counting function: } We shall use the following function for counting reciprocal classes of a given word length \(\ell\).

 {For a natural number $n$, the function $\Psi_{n}$ from $\mathbb{N}$ to $\mathbb{N}\cup \{0\}$, defined by $\Psi_n(x)$, is }the number of solutions of the equation, \begin{equation}\label{2.1}
		\sum_{i=1}^{n} p_{i}=x, 
	\end{equation}
	where all $p_{i}$  {for $i\in \{1, \ldots, n\}$,}  and $x$ are positive integers. It is easy to see that $\Psi_{n}(x) ={x-1 \choose n-1}$. 
	Let $\Psi_{n}^{r}(x)$ denote the number of solutions of \eqnref{2.1} with constraints $1\leq p_i \leq r$, for all $i\in \{1, \ldots, n\}$,  for some $r<x$. We have the following proposition.
	
	\begin{pro}
		If \( x < n \), then  $\Psi^r_n(x) =0$. If $x\geq n$, then we have 
		
		\[
		\Psi^r_n(x) = \sum_{k=0}^{\left\lfloor \frac{x - n}{r} \right\rfloor} (-1)^k \binom{n}{k} \binom{x - rk - 1}{n - 1}.\] 
	\end{pro}
	
	\begin{proof}
		 {If we choose $x<n$, there is no positive integral solution for the Equation \eqnref{2.1}. That implies $\Psi^r_n(x) =0$. }
        
        Now, we consider the Equation \eqnref{2.1} for $x\geq n$.
		Let \( S \) be the set of all solutions with \( p_i \ge 1\)  {for all $i\in \{1, \ldots, n\}$}, and let \( A_i \) be the set of solutions, where \( p_i \ge r+1 \). Let $S^{\prime}$ be the set of all required solutions. Then we have
		\[  S^{\prime} =S-\Big(\bigcup_{i}A_i\Big).  \]
		Then the number of required solutions is
		\[
		|S'| = |S| - \sum_i |A_i| + \sum_{i<j} |A_i \cap A_j| - \cdots + (-1)^k \sum_{i_1 < \dots < i_k} |A_{i_1} \cap \cdots \cap A_{i_k}| + \cdots,
		\]
		by the inclusion-exclusion principle. The set $A_{i_1} \cap \cdots \cap A_{i_k}$ is the set of the solutions of the Equation \eqnref{2.1} such that $p_{i_j}\geq r+1$ for all $j\in \{1,2,\dots,k\} $. To deduce the cardinality of this set, we need to rewrite the equation into the following equation:
        
		  \begin{equation}\label{new1}
		      		\sum_{i=1}^{n} p^{\prime}_i = x - rk \quad \text{(where all $p^{\prime}_i\geq 1$).}
\end{equation}
		
        The cardinality of the set $A_{i_1} \cap \cdots \cap A_{i_k}$, is the number of the solutions of the Equation \eqnref{new1},  
        i.e.,  $\Psi_n(x-rk)=\binom{x - rk - 1}{n - 1} $.

        There are \( \binom{n}{k} \) ways to choose such $p_i\geq r+1$.

		Hence, the total number is:
		\[
		\Psi^r_n(x) = |S^{\prime}|=\sum_{k=0}^{\left\lfloor \frac{x - n}{r} \right\rfloor} (-1)^k \binom{n}{k} \binom{x - rk - 1}{n - 1}.
		\]
		This completes the proof. 
	\end{proof}
	
	%In this section, the growth of the number of conjugacy classes of hyperbolic words in the Hecke group $\Gamma_{p}$.  %This section is inspired by the work  \cite{BV} where counting of the conjugacy classes in the modular group has been investigated. 
	%The actual value of $\Psi_{n}^{r}(t)$ is not known in general as $\Psi_{n}(t)$. But for some specific cases, it is known. \\
	\begin{lemma}\label{3.1}
		Let $r, x$ be given natural numbers. Then the total number of the solutions $\{n; k_1, \ldots, k_n\}$,~ $n\in\mathbb{N} ,~ {k}_i \in \mathbb Z$, $-r<k_{i}\leq r$ for all $k_{i}$ and $q$ be the number of $k_{i}$'s which are equal to $r$ of the equation $\sum_{i=1}^{n}|k_{i}|+n=x$ is given by the following sum after considering all possibilities for $ n$ and $k_{i}$'s, 
		\[\sum_{q=0}^{\lceil\frac{x}{r+1}\rceil-1}\sum_{n=\lceil \frac{x-q}{r}\rceil}^{\lfloor \frac{x-(r-1)q}{2}\rfloor}\Psi_{n}^{r-1}(x-n-rq)2^{n-q} {n\choose q}.\]
	\end{lemma}
	\begin{proof}
		Since some of the $k_{i}$'s are equal to $r$, we have two choices depending upon their sign for those $k_{i}$'s such that $k_{i}\neq r$. So, the number of solutions for the equation, for fixed $n$ and a fixed $q$ is:  
		\begin{equation}
			\Psi_{n}^{r-1}(x-n-qr)2^{n-q}.
		\end{equation}
		Now, $n$ will be maximum when all $k_{i} \neq r$ have absolute values $1$. Therefore,
		\[n-q+n+rq\leq x\]
		\[\imp 2n+(r-1)q\leq x\]
		\[\imp n\leq \frac{x-(r-1)q}{2}.\]
		Similarly, $n$ will reach the lower bound if the absolute value of each $k_{i} \neq r$ is $r-1$. Therefore,
		\[(n-q)(r-1)+rq+n\geq x\]
		\[\imp nr+q\geq x \]
		\[\imp n\geq \frac{x-q}{r}.\] So,  the number of all the solutions of the equation for fixed $q,$
		\[\sum_{n=\lceil \frac{x-q}{r}\rceil}^{\lfloor \frac{x-(r-1)q}{2}\rfloor}\Psi_{n}^{r-1}(x-n-rq)2^{n-q}.\]
		Note that 		$q$ has its minimum value $0$ when  all $k_i\neq r$, and $q$ has  maximum value $\lceil\frac{x}{r+1}\rceil-1$ when all $k_i$'s except one, are equal to $r$.   Also note that there are ${n\choose q}$ possibilities of having $q$ number of $k_i$ to be equal to $r$ in the tuple $(k_1,k_2,\dots ,k_n)$. 
		Now summing up for $q$ in between $0$ and $\lceil\frac{x}{r+1}\rceil-1$, the total number of solutions is given by: 
		\[\sum_{q=0}^{\lceil\frac{x}{r+1}\rceil-1}\sum_{n=\lceil \frac{x-q}{r}\rceil}^{\lfloor \frac{x-(r-1)q}{2}\rfloor}\Psi_{n}^{r-1}(x-n-rq)2^{n-q}{n\choose q}.\]
		This proves the lemma. 
	\end{proof}

	\section{Counting the reciprocal conjugacy classes}\label{3}
	To count the reciprocal classes in the Hecke group, we need to specify the representations of the reciprocal classes. We recall the following theorem from \cite{DG},
	\begin{theorem}{\cite{DG}}\label{th}
		For any reciprocal hyperbolic element in $ \Gamma_{2r} $, the following possibilities can occur.

			\begin{enumerate}
				\item Suppose in the reciprocal class, the reciprocators are conjugate to $ \iota $. Then, there are exactly four symmetric elements in that class.
				\item Suppose in the reciprocal class, the reciprocators conjugate to $\gamma_{p}^{r} $.  {Then, the number of $p$-reciprocal elements is exactly $4r$. }
				\item Suppose in the reciprocal class, each reciprocal element has two types of reciprocators; one type of reciprocator conjugates to $ \iota $, and another type of reciprocator conjugates to $ \gamma_{p}^{r} $.
				\begin{enumerate}
					
					\item  {If the reciprocal class does not contain any nonzero power of $ \iota \g^r$, then the number of symmetric elements and the number of $p$-reciprocal elements contained in the class are exactly two and exactly $2r$, respectively.}

					\item If the reciprocal class contain any non-zero power of $ \iota \g^r$, then 
                     {the number of symmetric $p$-reciprocal elements and the number of $p$-reciprocal elements contained in the class are exactly two and exactly $2r-2$, respectively}.			\end{enumerate}
			\end{enumerate}

	\end{theorem}
	
	\subsection{Counting reciprocal classes}
	We have seen that a reciprocal class $[\alpha]$ in $\Gamma_p$ can be divided into three types depending on whether it has all the reciprocators conjugate to only one of  $\iota$ and  $\gt$, or both of them. In the following, we obtain a `  {normal} form' for such reciprocal classes. 
	Note that a reciprocal element $\alpha$ with the form  $ \io \gt\io \g^{k_{1}}\io \g^{k_{2}}\dots \io \g^{k_{n}}\io \gt\io \g^{-k_{n}}\dots\io \g^{-k_{2}}\io \g^{-k_{1}} $, is a symmetric $p$-reciprocal element, but any even power of $\alpha$ i.e. $\alpha^{2m}$ for  $m\in \z$, may not remain in the same form. It can take the form of a reduced expression like  { $\io K\io K^{-1}$ or  $\gt K\gt K^{-1}$}, up to conjugacy. Although reciprocal elements have distinct forms in the primitive case, they may coincide in the general case.
	
    \begin{lemma}\label{lem}
Let $[\tilde{\alpha}]$ be a reciprocal class in $\Gamma_{2r}$. Then there exists a cyclically reduced representative $\alpha$ (or $\alpha^{-1}$) in this class such that $\alpha$ admits one of the following   {normal} forms.

In all cases, $n \in \mathbb{N}$, the integers $k_i$ satisfy 
\(
-r < k_i \leq r \quad \text{for all } i,
\)
and the expressions are cyclically reduced. Moreover, the representative $\alpha$ is one of the two forms up to cyclic permutation for the first two cases.

\begin{enumerate}
\item Suppose one of the reciprocators of $\alpha$ is $\io$ and $\io \alpha$ is conjugated to $\io$. Then $\alpha$ has the form
\[
\io \g^{k_{1}}\io \g^{k_{2}}\dots \io \g^{k_{n}}\io \g^{-k_{n}}\dots \io \g^{-k_{2}}\io \g^{-k_{1}}.
\]

\item Suppose one of the reciprocators of $\alpha$ is $\io\gt\io$ and $\io\gt\io \,\alpha$ is conjugated to $\gt$. Then $\alpha$ has the form
\[
\io \gt\io \g^{k_{1}}\io \g^{k_{2}}\dots \io \g^{k_{n}}\io \gt\io \g^{-k_{n}}\dots \io \g^{-k_{2}}\io \g^{-k_{1}}.
\]

\item Suppose $\alpha$ admits two types of reciprocators, one of which is $\io$ with $\io\alpha$ conjugate to $\gt$.

\begin{itemize}
\item If the reciprocal class $[\alpha]$ does not contain any non-trivial power of $\io\gt$, then $\alpha$ has one of the two forms
\[
\io\gt\io \g^{k_{1}}\io \g^{k_{2}}\dots \io \g^{k_{n}}\io \g^{-k_{n}}\dots \io \g^{-k_{2}}\io \g^{-k_{1}},
\]
or
\[
\io \g^{k_{1}}\io \g^{k_{2}}\dots \io \g^{k_{n}}\io\gt\io \g^{-k_{n}}\dots \io \g^{-k_{2}}\io \g^{-k_{1}}.
\]

\item If $[\alpha]$ contains a non-trivial power of $\io\gt$, then it has a unique cyclically reduced representative of the form
\[
(\io\gt)^k, \quad k \in \mathbb{Z}\setminus\{0\},
\]
which is unique up to cyclic permutation.
\end{itemize}
\end{enumerate}
\end{lemma}
		\begin{proof}
		For proving (1), assume that the reciprocal word $ R $ from the reciprocal class $  [\alpha]$  be a symmetric element:  $R=\io \eta \io\eta^{-1}$, where $ \eta \in \Gamma_p $. There are exactly three other symmetric elements: $R^{-1}$,  
		$ \eta^{-1} R\eta$ and $\eta^{-1} R^{-1}\eta$. We can write $ \eta $ as a word in terms of $ \io$ and $ \g $. If $ \eta_{1}= \io\g^{k_{1}}\io\g^{k_{2}}\dots \io\g^{k_{n}}$ where $k_i$'s are ranging from $-r+1$ to $r$, then there are four choices of $ \eta $ either $ \eta_{1}, \io\eta_{1}, \eta_{1}\io \text{ or } \io \eta_{1}\io. $\n Our claim is that any  two of the elements, $R,~R^{-1}$,  
		$ \eta^{-1} R\eta$ and $\eta^{-1} R^{-1}\eta$ provide the required forms for the choices of $\eta$. 
		We first assume $ \eta=\eta_{1} $. Therefore,  \[R=\io\io\g^{k_{1}}\io\g^{k_{2}}\dots \io\g^{k_{n}}\io (\io\g^{k_{1}}\io\g^{k_{2}}\dots \io\g^{k_{n}})^{-1}\]
		\[=\g^{k_{1}}\io\g^{k_{2}}\dots \io\g^{k_{n}}\io\g^{-k_{n}}\io\dots\g^{-k_{2}}\io\g^{-k_{1}}\io.\]
		The above expression implies that $R^{-1}$ is in the desired form. Also note that $\eta ^{-1}R^{-1}\eta$ is in similar form: $\io \g^{-k_{n}}\dots\io  \g^{-k_{2}}\io \g^{-k_{1}}\io\g^{k_{1}}\io \g^{k_{2}}\dots \io \g^{k_{n}}$. Thus, when  $ \eta=\eta_{1} $, the desired form of $\alpha$ is given by  either $ R^{-1}=\eta\io\eta^{-1}\io $ or $ \eta^{-1}R^{-1}\eta  = \io\eta^{-1}\io\eta $.
		
		Similarly, for the other choices of $\eta$, the desired form of $\alpha$ may be chosen as follows.
		
		When $\eta=\io \eta_1$, $\alpha$ is either $R$ or $\eta^{-1} R^{-1}\eta$; 
		\\
		for $ \eta=\eta_{1}\io $, $ R=\io(\eta_{1}\io)\io(\eta_{1}\io)^{-1}=\io \eta_{1}\io\eta_{1}^{-1}. $ So, similar to the first case, $\alpha$ has two choices $ R^{-1} $ and $ \eta^{-1}R\eta $.
		
		for $ \eta=\io\eta_{1}\io, $ $ R=\io\io\eta_{1}\io\io\io\eta_{1}^{-1}\io=\eta_{1}\io\eta_{1}^{-1}\io .$ So, similar to the second case, $\alpha$ have two choices, $ R $ and $ \eta^{-1}R\eta. $

		For $(2)$, assume that $R$ is a  $p$-reciprocal element from $[\alpha].$ By \thmref {th}, there are $4r$ many $p$-reciprocal elements, but we choose the cyclically reduced one as $R$. Let $R=\gt\eta \gt\eta^{-1}$.  Then, again by similar arguments as  in (1), if we choose, $ \eta= \io\g^{k_{1}}\io\g^{k_{2}}\dots \io\g^{k_{n}}$, then 
		\[R=\gt \io\g^{k_{1}}\io\g^{k_{2}}\dots \io\g^{k_{n}}\gt \g^{-k_{n}}\io\g^{-k_{(n-1)}}\dots \io\g^{-k_{1}}\io\]
		\[\imp R=\gt \io\g^{k_{1}}\io\g^{k_{2}}\dots \io \g^{k_{n-1}}\io\gt\io\g^{-k_{n-1}}\dots \io\g^{-k_{1}}\io\]  
		\[\imp \io R\io =  \io \gt\io \g^{k_{1}}\io \g^{k_{2}}\dots \io \g^{k_{n}}\io \gt\io \g^{-k_{n}}\dots\io \g^{-k_{2}}\io \g^{-k_{1}}=R', \hbox{say}.\]	Thus  $R'$ is in the form of $\io \gt\io \g^{-k_{n}}\dots\io \g^{-k_{2}}\io \g^{-k_{1}} \io \gt\io \g^{k_{1}}\io \g^{k_{2}}\dots \io \g^{k_{n}}  $.  And $R''= \io\eta^{-1} R\eta \io$ is also in the same form. Since $R$ and $\eta^{-1}R\eta$ are $ p $-reciprocal elements, in both cases $ \g^{k}R\g^{-k} $ and $ \g^{k} \eta^{-1}R\eta\g^{-k} $  for $-r<k_i\leq r$, are also $ p $-reciprocal elements that completes the list of $ p $-reciprocal elements which are not cyclically reduced. So, we have two choices of $\alpha$ from the cyclically reduced words:  $ R' , ~R'',R'^{-1}, R''^{-1}$. For other choices of $\eta$ as noted above,  the arguments will be similar as in the proof of $(1)$. This proves (2).
		
		For $(3)$, by \thmref{th}, there are $2r+2$ choices for $R$. It is easy to see that only two of them are in the desired form:  $ R=\gt\eta\io \eta^{-1} $  and $ R'=\eta^{-1}\gt \eta \io.$ The  others are either $ \g^{k}R\g^{-k} $ or $\io R'\io$ where $-r<k\leq r$. Let $\eta_1$ as  in  $(1)$, so the choices of $\eta$ are, $\eta_1$, $\io\eta_1$, $\eta_1 \io$, and $\io\eta_1\io$.  
		If we assume $\eta=\eta_1$, then $\io R\io= \io\gt\eta\io\eta^{-1} $ and $R'=\eta^{-1}\gt\eta\io $ are the choices for $\alpha$. 
		By similar arguments as in  $ (1) $, we obtain the two desired forms for the remaining choices of $\eta$. This proves $(3)$.
	\end{proof}

\medskip Now we count the reciprocal classes by expressing their representatives in terms of reduced words and their word lengths.

	\begin{lemma}\label{3.3}
		Consider the set of all reciprocal classes in $\Gamma_{2r}$ with length $2l$ and reciprocators conjugate to $\iota$. Then, this set has cardinality
		% $ \frac{1}{2}\sum_{n=\lceil\frac{l}{r+1}\rceil}^{\lfloor l/2 \rfloor}\Psi_{n} ^{r}(l-n)2^{n} $ if $ p $ is $2r+1$, otherwise
		$$\frac{1}{2}\sum_{q=0}^{\lceil\frac{l}{r+1}\rceil-1} \sum_{n= \lceil\frac{l-q}{r}\rceil }^{\lfloor \frac{l-(r-1)q}{2}\rfloor}\Psi_{n}^{r-1}(l-n-rq)2^{n-q} {n \choose q}.$$ 
	\end{lemma}
	\begin{proof}
		%	\tb{For the odd case:}
		%If $ p=2r+1 $ Every reciprocal class has a cyclically reduced word $ R= \io\g^{k_{1}}\io \g^{k_{2}}\dots \io \g^{k_{n}}\io \g^{-k_{n}}\dots\io \g^{-k_{2}}\io \g^{-k_{1}} $. Now we will calculate the number of reciprocal classes for each word length $ 2l $.
		%\[ \hbox{ length of word $ R $}=|R|=2\sum_{i=1}^{n} k_{i} +2n.\]
		%Let\begin{equation}\label{label}
			%	 2\sum_{i=1}^{n} k_{i} +2n=2l 
			%\end{equation}
			%\[\imp \sum_{i=1}^{n} k_{i} =l-n.\]
			%So, we just need to calculate the number for integral solutions for each $ k_{i}\leq r $.
			%Let us consider all $ k_{i} $'s  are non-zero then the number of solution \ref{label} is $ \Psi_{n} %^{r}(l-n) $(say). So, the number of all reciprocal classes of length $ 2l $ is $ \frac{1}{2}\sum_{n=\lceil\frac{l}{r+1}\rceil}^{\lfloor l/2 \rfloor}\Psi_{n} ^{r}(l-n)2^{n} $.
			%\tb{For the even case:}
			Consider the symmetric reciprocal word $ R= \io\g^{k_{1}}\io \g^{k_{2}}\dots \io \g^{k_{n}}\io \g^{-k_{n}}\dots\io \g^{-k_{2}}\io \g^{-k_{1}} $, where, $-r< k_{i}\leq r $ for all $1\leq i\leq n$. We see that 
			\[ \g^{-k_{n}}\io \dots \g^{-k_{1}}\io R\io\g^{k_{1}}\io \g^{k_{2}}\dots \io \g^{k_{n}}=\io \g^{-k_{n}}\dots\io \g^{-k_{2}}\io \g^{-k_{1}}\io\g^{k_{1}}\io \g^{k_{2}}\dots \io \g^{k_{n}} .\]
			Let  $R'=\io \g^{-k_{n}}\dots\io \g^{-k_{2}}\io \g^{-k_{1}}\io\g^{k_{1}}\io \g^{k_{2}}\dots \io \g^{k_{n}}$. Then 
			$ R $ and $ R' $ are of the form mentioned in \lemref{lem}. Since, given word length of $ R $ and  $ R' $ is $ 2l $, we have 
			\[\sum^n_{i=1} 2|k_i|+2n=2l ~\imp ~ \sum_{i=1}^{n} |k_{i}|=l-n, 		\hbox{ where } -r<k_{i}\leq r. \]
			Let $ q $ be the number of $ k_{i} $'s such that $k_i=r$. After rewriting the indices in the above equation, we get \[ \sum_{i=1}^{n-q} |k_{i}|=(l-n-rq). \]
			Then, by the similar arguments from \lemref{3.1},
			%\[\imp (l-n-rq)\leq (n-q)(r-1)\]
			%\[\imp l-n-rq\leq n(r-1)-qr+q\]
			%\imp l-q\leq nr\]
			%\[\imp \frac{l-q}{r}\leq n.\]
			the lower bound of $ n $ is $ \lceil\frac{l-q}{r}\rceil $ and 
			%When we add up all $ 1 $'s then this is the upper bound for $ n $, 
			%\[n-q=l-n-rq\]
			%\[\imp 2n=l-(r-1)q\]
			%\[\imp n=\frac{l-(r-1)q}{2}.\]
			the upper bound for $ n $ is $ \lfloor \frac{l-(r-1)q}{2}\rfloor $. Also, $ q $ varies from $ 0 $ to $\lceil\frac{l}{r+1}\rceil-1$, excluding the symmetric $ p $-reciprocal element $ (\io \gt)^{m} $ where, $ m\in \mathbb{Z}\setminus\{0\} $. Now, using \lemref{lem}, \lemref{3.1} we obtain that the number of the reciprocal symmetric classes is  $$\frac{1}{2}\sum_{q=0}^{\lceil\frac{l}{r+1}\rceil-1} \sum_{n= \lceil\frac{l-q}{r}\rceil }^{\lfloor \frac{l-(r-1)q}{2}\rfloor}\Psi_{n}^{r-1}(l-n-rq)2^{n-q} {n \choose q}.$$
			The lemma is proved. 
		\end{proof}
		
		\medskip 
		%Note that, other type of reciprocal classes will be there only when $  p $ is even. So, to find the total number of reciprocal classes with a certain length, we also need to consider other cases. In the following lemma, we have found the number of conjugacy classes of $ p $-reciprocal elements.
		\begin{lemma}\label{lem0}
			Consider the set of all  $p$-reciprocal classes $\Gamma_{2r}$ with length $2l$. Then, this set has cardinality
			$$  \frac{1}{2}\sum_{q=0}^{(\lceil\frac{l}{r+1}\rceil-2)} \sum_{n=\lceil \frac{l-(r+1)-q}{r}\rceil}^{\lfloor\frac{l-1-(r+1)q-r}{2}\rfloor} \Psi_{n}^{r-1}(l-(n+1)-(q+1)r)2^{n-q} {n\choose q}.$$
			
		\end{lemma}
		\begin{proof}
			Let $R$ be a $p$-reciprocal element of this form: $ \io \gt\io \g^{k_{1}}\io \g^{k_{2}}\dots \io \g^{k_{n}}\io \gt\io \g^{-k_{n}}\dots\io \g^{-k_{2}}\io \g^{-k_{1}}$, see, \lemref{lem}.
			Therefore, we rewrite the \(k_i\) with absolute value less than \(r\), and let \(q\) be the number of \(k_i\) that are equal to \(r\).
			\[2\sum^{n-q}_{i=1}k_i+2n+2rq+2(r+1)=2l\]
			\begin{equation}
				\imp	\sum_{i=1}^{n-q} |k_{i}|+n+rq+ r+1=l, 
			\end{equation} where $ {|k_i|}\leq r-1$. Then, 
			\[ \sum_{i=-1}^{n-q}{|k_{i}|}+(n+1)+r(q+1)=l\]
			\[\imp \sum_{i=1}^{n-q} |k_{i}|=l-r(q+1)-(n+1).\]
			So, with the help of the last expression, we can deduce the total number of conjugacy classes. The remaining part is to show the upper and lower limits of $ n $ in the sum. 
			For the lower limit, 
			\[\sum_{i=1}^{n-q}|k_{i}|\leq (n-q)(r-1)\]
			\[\imp l-r(q+1)-(n+1) \leq (n-q)(r-1)\]
			\[\imp  l-r(q+1)-(n+1) \leq nr-n-qr+q \]
			\[\imp l-(r+1)-q\leq nr\]
			\[\imp \frac{l-(r+1)-q}{r}\leq n.\]
			So, the lower bound of $ n $ is $ \lceil \frac{l-(r+1)-q}{r}\rceil .$
			And the upper bound when all $ k_{i} $'s have the lowest value, which is $ 1 $, i.e.,
			\[n-q=l-(n+1)-r(q+1)\]
			\[\imp n=\frac{l-1-(r+1)q-r}{2}.\]
			So, the upper bound of $n$ is $ \lfloor\frac{l-1-(r+1)q-r}{2}\rfloor. $ And, $q$ varies from $0$ to $\lceil\frac{l-r-1}{r+1}\rceil-1=\lceil\frac{l}{r+1}\rceil-2$.  Then by \lemref{3.1},\lemref{lem}, the number of $ p $-reciprocal classes where the reciprocators are conjugate of $\gt$, is 
			$$  \frac{1}{2}\sum_{q=0}^{(\lceil\frac{l}{r+1}\rceil-2)} \sum_{\lceil \frac{l-(r+1)-q}{r}\rceil}^{\lfloor\frac{l-1-(r+1)q-r}{2}\rfloor} \Psi_{n}^{r-1}(l-(n+1)-(q+1)r)2^{n-q} {n\choose q}.$$
			This completes the proof.  
		\end{proof}
		\begin{lemma}\label{lem1}
			The number of symmetric $p$-reciprocal classes that do not contain any non-trivial power of $ \io\gt $ of word length $ 2l $, resp.  $2l+1$,  is $$\frac{1}{2}\sum_{q=0}^{\lceil\frac{l-u}{r+1}\rceil-1} \sum_{n= \lceil\frac{l-u-q}{r}\rceil }^{\lfloor \frac{l-u-(r-1)q}{2}\rfloor}\Psi_{n}^{r-1}(l-u-n-rq)2^{n-q} {n \choose q} $$ where, $r=2u $, resp. $r=2u+1$.
		\end{lemma}
		\begin{proof}
			Suppose $R$ is the cyclically reduced representation of the mentioned reciprocal class, then \begin{equation} R= \io \gt\io \g^{k_{1}}\io \g^{k_{2}}\dots \io \g^{k_{n}}\io \g^{-k_{n}}\dots\io \g^{-k_{2}}\io \g^{-k_{1}}.\end{equation} 
			Here, we first consider  $r=2u$. Then
			\begin{equation}
				r+1+ 2\sum_{i=1}^{n} |k_{i}| +2n=2l+1
			\end{equation} where $-r< k_{i}\leq r$ for all $i$. Then implies,
			\[	 2\sum_{i=1}^{n} |k_{i}| +2n=2(l-u) \]
			\[\imp 	 \sum_{i=1}^{n} |k_{i}| +n=l-u. \]
			Also, $R$ is conjugate to $\io  \g^{k_{1}}\io \g^{k_{2}}\dots \io \g^{k_{n}}\io\gt\io \g^{-k_{n}}\dots\io \g^{-k_{2}}\io \g^{-k_{1}}$, so there are two choices of words for the required form.
			Now, the process is the same as the symmetric element case. So, the number of the reciprocal class is $\frac{1}{2}\sum_{q=0}^{\lceil\frac{l-u}{r+1}\rceil-1} \sum_{n= \lceil\frac{l-u-q}{r}\rceil }^{\lfloor \frac{l-u-(r-1)q}{2}\rfloor}\Psi_{n}^{r-1}(l-u-n-rq)2^{n-q} {n \choose q} $.

			Similarly, for the odd case, if we choose $u=\lfloor \frac{r-1}{2}\rfloor$, we will get the same number of conjugacy classes. 		\end{proof}

		\medskip 		The only remaining case when the reciprocal class is represented by a non-trivial power of $\io\gt $, follow from \lemref{lem}. Combining all these, we have the following theorem. 
		\begin{thm}\label{count1}
			Let $\Gamma_{2r}$ be the Hecke group with the presentation $\langle \io,\g~|~\io^{2},\g^{2r}\rangle$. Then the following statements hold:
			\begin{enumerate}
				\item The number of symmetric reciprocal classes of word length $2l$ is, 	$$\frac{1}{2}\sum_{q=0}^{\lceil\frac{l}{r+1}\rceil-1} \sum_{n= \lceil\frac{l-q}{r}\rceil }^{\lfloor \frac{l-(r-1)q}{2}\rfloor}\Psi_{n}^{r-1}(l-n-rq)2^{n-q} {n \choose q}.$$ 
				\item The number of $p$-reciprocal classes of word length $2l$ is,
				$$  \frac{1}{2}\sum_{q=0}^{(\lceil\frac{l}{r+1}\rceil-2)} \sum_{n=\lceil \frac{l-(r+1)-q}{r}\rceil}^{\lfloor\frac{l-1-(r+1)q-r}{2}\rfloor} \Psi_{n}^{r-1}(l-(n+1)-(q+1)r)2^{n-q} {n\choose q} .$$
				
				\item The number of symmetric $ p$-reciprocal classes of word length $2l$, resp. $2l+1$ depending on $r$ is,
				$$\frac{1}{2}\sum_{q=0}^{\lceil\frac{l-u}{r+1}\rceil-1} \sum_{n= \lceil\frac{l-u-q}{r}\rceil }^{\lfloor \frac{l-u-(r-1)q}{2}\rfloor}\Psi_{n}^{r-1}(l-u-n-rq)2^{n-q} {n \choose q} +1$$ where, $r=2u $, resp. $r=2u+1.$
			\end{enumerate}
		\end{thm}
		\begin{proof}
			The proof of this theorem follows from \lemref{3.3}, \lemref{lem0} and \lemref{lem1}.
		\end{proof}
		\section{Spherical growth  of the Reciprocal Classes}\label{4}

 Let $\mathcal{N}_\ell$ denote the number of conjugacy classes of reciprocal elements of word length $\ell$.
\begin{lemma}
Let $[\alpha]$ be a reciprocal conjugacy class in $\Gamma_{2r}$.

\begin{enumerate}
\item If $r$ is odd, then $[\alpha]$ has even word length.
\item If $r$ is even, then $[\alpha]$ has odd word length if and only if it is symmetric $p$-reciprocal. In particular, the symmetric and $p$-reciprocal classes have even word length.
\end{enumerate}
\end{lemma}

\begin{proof}
By Lemma~\ref{lem}, every reciprocal conjugacy class admits a cyclically reduced representative of one of the following three types: symmetric, $p$-reciprocal, or symmetric $p$-reciprocal.

In the symmetric case, the representative has the form
\[
\alpha=
\iota \gamma_r^{k_1}\iota \gamma_r^{k_2}\cdots \iota \gamma_r^{k_n}
\iota \gamma_r^{-k_n}\cdots \iota \gamma_r^{-k_2}\iota \gamma_r^{-k_1},
\]
and hence
\[
\ell(\alpha)=2\sum_{i=1}^n |k_i|+2n,
\]
which is always even.

In the $p$-reciprocal case, by Lemma~\ref{lem0}, one has
\[
\ell(\alpha)=2\sum_{i=1}^n |k_i|+2n+2(r+1),
\]
which is again always even.

In the symmetric $p$-reciprocal case, by Lemma~\ref{lem1}, one has
\[
\ell(\alpha)=2\sum_{i=1}^n |k_i|+2n+(r+1).
\]
Therefore the parity of $\ell(\alpha)$ is exactly the parity of $r+1$.

If $r=2u+1$ is odd, then $r+1=2u+2$ is even, so in all three cases the word length is even.

If $r=2u$ is even, then $r+1=2u+1$ is odd, so the symmetric $p$-reciprocal case has odd word length, while the other two cases remain even.

This proves both assertions.
\end{proof}

\medskip

\subsection{Proof of Theorem~\ref{mthm}.}
In order to  prove the  Theorem  we shall show that $\mathcal N_\ell$ satisfies a linear recurrence whose growth is determined by the unique positive root $\rho \in [\sqrt{2},2]$ of
\[
P(x)=x^{r+1}-2\sum_{j=1}^{r-1} x^{\,r-j}-1.
\]
 We distinguish two cases depending on the parity of $r$.
\subsubsection{When $r$ is odd}
	  We first consider the case where $r=2u+1$ is odd. By the preceding lemma, every reciprocal conjugacy class has even word length. Hence, the counting problem reduces to the study of the sequence $\mathcal N_{2l}$.
		\begin{lemma}\label{count}
			Let $\mathcal{N}_{2l}$ be the set of reciprocal classes of word length $2l$, then the following statements hold.
			%$\io\g^{k_{1}}\io\g^{k_{2}}\dots \io\g^{k_{n}}\io \g^{-k_{n}}\dots \io\g^{-k_{1}}$.
			\begin{enumerate}
				\item If $l\leq r$, then the number of reciprocal classes is $ \frac{1}{6}(2^{l}+2(-1)^{l})$.
				\item If $l=r+1$, the number of reciprocal classes is $\frac{1}{6}(2^{l}+2(-1)^{l})+\frac{ 1}{6}(2^{u+1}+2(-1)^{u+1}).$
				\item If $l\geq r+2$, 
				\[|\mathcal{N}_{2l}|=2\sum_{j=1}^{r-1}|\mathcal{N}_{2(l-j-1)}|+|\mathcal{N}_{2(l-r-1)}|.\]
			\end{enumerate} 
		\end{lemma}
			\begin{proof}
		For $(1)$, it is given that  $l\leq r$.  There are only symmetric reciprocal classes; otherwise, the length will exceed $2r$.
		So, the conditions to determine the cardinality of $\mathcal{N}_{2l}$ for $\Gamma_{2r}$, is the same as the cardinality of $\mathcal{N}_{2l}$ for $\Gamma_{2r-1}$. Since $2r-1$ is odd, it follows from \cite[Theorem 5.3.2.]{Ma} that the cardinality of the cyclically reduced reciprocal words is 
		\[\frac{1}{3}(2^{l}+2(-1)^{l}).\]
	Since each reciprocal class has two cyclically reduced reciprocal $\io-\g$ words from \lemref{lem}, this proves $(1)$.			
		
		For $(2)$,
		again, there is no $p$-reciprocal class in $\mathcal{N}_{2l}$; otherwise the word length will exceed $2(r+1).$ So, the conjugacy classes are either symmetric reciprocal classes or primitive symmetric $p $-reciprocal classes. Let $R$ be the representative of a conjugacy class in $\mathcal{N}_{2l}$ as mentioned in the proof of the \lemref{lem}. Then, \[R=(\io\gt)^{\epsilon}\io  \g^{k_{1}}\io \g^{k_{2}}\dots \io \g^{k_{n}}\io \g^{-k_{n}}\dots\io \g^{-k_{2}}\io \g^{-k_{1}}, \] where $\epsilon$ is either $0$ or $1$, depends on the respective choice of symmetric reciprocal class or symmetric $ p$-reciprocal class.  In each case, this gives us the following equation,		\begin{equation}
			2\sum_{i=1}^{n}|k_{i}| +2n+\epsilon(r+1)= 2l.
		\end{equation}
		Let us consider $q$, as the number of $ k_i$'s such that they are equal to $r.$ So,  after rewriting the $k_{i}$'s, we have,
		\begin{equation}
			2\sum_{i=1}^{n-q}|k_{i}| +2n+2qr+\epsilon(r+1)= 2l
		\end{equation} where $k_{i}\leq r-1,$ $q\geq 0$, $n>0$, and $r=2u+1. $ 
		This implies \[2\sum_{i=1}^{n-q}|k_{i}| +2n+2qr+\epsilon(r+1)= 2(r+1).\]
		Here, three disjoint cases may arise for the choices of $\epsilon~$and$~q$.
		
		\medskip  \tb{Case 1:}	If $\epsilon=0$ and $q=0$,
		then the element will belong to symmetric classes in $\mathcal{N}_{2l}.$ In that case, we can get $|\mathcal{N}_{2l}|$ from \cite[Theorem 5.3.2.]{Ma}, i.e., $\frac{ 1}{6}(2^{l}+2(-1)^{l})-1$ by using the same fact as $(1)$. 
		
		\tb{Case 2:} 	If	$\epsilon=1$ and $q=0$, this implies \[2\sum_{i=1}^{n}|k_{i}| +2n+(r+1)= 2(r+1)\]
		\[\imp 2\sum_{i=1}^{n}|k_{i}| +2n=(r+1). \]
		Therefore, $\sum_{i=1}^{n}|k_{i}| +n=(u+1)$, 
		where $k_{i}\leq r-1$ and  $u=\frac{r-1}{2}.$ Since $u+1\leq r$, we obtain, $\frac{ 1}{6}(2^{u+1}+2(-1)^{u+1})$ from $(1)$.

		\tb{Case 3:} If $q=1$ and $\epsilon=1,$ then
		\[ 2\sum_{i=1}^{n-1}|k_{i}| +2n+2r= 2(r+1)\]
		\[\imp\sum_{i=1}^{n-1}|k_{i}| +n+r= (r+1)\]
		\[\imp \sum_{i=1}^{n}|k_{i}|=(1-n). \]
		Since the absolute values are non-negative, $n$ and $q$ must be $1$, and all $k_{i}$'s are $0.$ Then $R=\io\gt\io\gt$. So, there is only one element for the case $q>0.$		
		Then the required number is $$\frac{ 1}{6}(2^{l}+2(-1)^{l})+\frac{ 1}{6}(2^{u+1}+2(-1)^{u+1}).$$
		It proves $(2)$.
		
		For $(3)$, We emphasize that the three reciprocal types are not, in general, mutually disjoint: a reciprocal class may contain both symmetric and \(p\)-reciprocal elements. Nevertheless, the corresponding primitive representatives are disjoint. Therefore, we decompose \(\mathcal{N}_{2l}\) into three disjoint subsets \(A_{2l}, B_{2l}, C_{2l}\), where \(A_{2l}\), \(B_{2l}\), and \(C_{2l}\) denote the sets of conjugacy classes whose primitive representatives are symmetric reciprocal, \(p\)-reciprocal, and symmetric \(p\)-reciprocal, respectively. We can see that the set $A_{2l}$ is a collection of the reduced words, which are symmetric and of the form $  \io\g^{k_{1}}\io \g^{k_{2}}\dots \io \g^{k_{n}}\io \g^{-k_{n}}\dots\io \g^{-k_{2}}\io \g^{-k_{1}} $, where $-r<k_{i}\leq r$. Then, the set $A_{2l,j}=\io\g^{j}\io \g^{k_{2}}\dots \io \g^{k_{n}}\io \g^{-k_{n}}\dots\io \g^{-k_{2}}\io \g^{-j}$  gives the equation,
		\[ 2+2|j|+2\sum_{i=2}^{n}|k_{i}|+2n=2l,\]
		where $-r<k_i,~ j\leq r$ .
		This implies \[\sum_{i=2}^{n}|k_{i}|+n=l-|j|-1.\]
		So, $|A_{2l,j}|=|A_{2(l-|j|-1)}|$. Now, varying $|j|$ over the set $ \{r,r-1,\dots, 1\} $, we get the following condition hold,
		\[|A_{2l}|=2\sum_{j=1}^{r-1}|A_{2(l-|j|-1)}|+|A_{l-r-1}|.\]
		Here, $ 2 $ appears for both signs of $k_{i}$'s.
		Similarly, we have the following equalities.
		$$|B_{2l}|=2\sum_{j=1}^{r-1}|B_{2(l-j-1)}|+|B_{l-r-1}|, ~~ |C_{2l}|=2\sum_{j=1}^{r-1}|C_{2(l-j-1)}|+|C_{l-r-1}|.$$ 
		Therefore, we have
		\[|\mathcal{N}_{2l}|=2\sum_{j=1}^{r-1}|\mathcal{N}_{2(l-j-1)}|+|\mathcal{N}_{l-r-1}|.\]
		This proves $(3)$. 
	\end{proof}

\medskip 

Now, we recall the following well known results. 
\begin{thm}[\textbf{Perron}]\label{perron}
The polynomial 
\[
f(x) = x^n + a x \pm 1
\]
is irreducible for $|a| \geq 3$. When $|a| = 2$, $f(x)$ is either irreducible or has a factor $x \pm 1$. In the latter case, the second factor of $f(x)$ is irreducible.
\end{thm}

		\begin{theorem}\label{4.2}
			Every real root of a monic polynomial with integer coefficients is either an integer or irrational.
		\end{theorem}
		
		\begin{defn}\cite[Definition 8.1.2]{An}
			{Let $f(z)=z^n+a_{n-1}z^{n-1}+a_{n-2}z^{n-2}+\dots +a_1 z+a_0$ be a monic  with complex coefficients. Then, the \textit{Cauchy bound} of $f$, say $\xi$, is the unique positive root of the equation $z^n-\sum_{k=0}^{n-1}\lvert a_k\rvert z^k$=0}. 
		\end{defn}
		\begin{thm}[\textbf{Cauchy}]\label{4.4}
			Let  $f(z)=z^n+a_{n-1}z^{n-1}+a_{n-2}z^{n-2}+\dots +a_1 z+a_0$  be an $n$-degree polynomial with complex coefficients. Then all zeros of $f(z)$ have modulus less than or equal to the Cauchy bound $\xi.$
		\end{thm}
        
	We use these results to analyze the characteristic equation corresponding to
\[
|\mathcal{N}_{2l}|=2\sum_{j=1}^{r-1}|\mathcal{N}_{2(l-j-1)}|+|\mathcal{N}_{2(l-r-1)}|,
\]
namely,
\[
x^l=2\sum_{j=1}^{r-1}x^{l-j-1}+x^{l-r-1}.
\]
For \(x\neq 0\), we have the following equation: $x^{r+1}-2\sum_{j=1}^{r-1}x^{r-j}-1$. Let $$P(x)=x^{r+1}-2\sum_{j=1}^{r-1}x^{r-j}-1.$$
Since $x=-1$ is a root of $P(x)$, then
\[
P(x)=(x+1)Q(x),
\qquad
Q(x):=x^r-x^{r-1}-\cdots-x-1.
\]

Let \(\widetilde Q(x)=x^rQ(1/x)=1-x-\cdots-x^r\) be the reciprocal polynomial of \(Q\). Then
\[
-(x-1)\widetilde Q(x)=x^{r+1}-2x+1.
\]
By Theorem~\ref{perron}, the polynomial \(x^{r+1}-2x+1\) has the factor \(x-1\), and its other factor \(\widetilde Q(x)\) is irreducible over \(\mathbb Q\). Hence \(Q(x)\) is irreducible over \(\mathbb Q\). Since we are in characteristic \(0\), \(Q\) has no repeated roots; moreover, \(Q(-1)\neq 0\), so \(P(x)=(x+1)Q(x)\) also has no repeated roots.

By Descartes' rule of signs, \(Q\) has exactly one positive real root. That is also a root of $P$. We now show that this root lies in \((\sqrt2,2)\) and is the unique root of maximal modulus.
		\begin{lemma}\label{2.7}
The recurrence relation
\begin{equation}\label{2.71}
|\mathcal{N}_{2l}|=2\sum_{j=1}^{r-1}|\mathcal{N}_{2(l-j-1)}|+|\mathcal{N}_{2(l-r-1)}|
\end{equation}
has a characteristic polynomial with a unique positive real root $\rho$ lying in the interval $(\sqrt{2},2)$.

Moreover, 
\(
|\mathcal{N}_{2l}|\sim C_1\rho^{l} \quad \text{as } l\to\infty,
\)
where $\rho$ is the root of the recurrence polynomial with the largest absolute value and $C_1$ is a constant.
\end{lemma}
		
		\begin{proof}
			 If we define $P$ as a map from real line to real line, then $P(\sqrt{2})<0$: 
			\[ P(\sqrt{2})=2^{\frac{r+1}{2}}-2\sum_{j=1}^{r-1}2^{\frac{r-j}{2}}-1\]
			\[=2^{\frac{r+1}{2}}-2^{\frac{r+1}{2}}-2\sum_{j=2}^{r-1}2^{\frac{r-j}{2}}-1\]
			\[=-2\sum_{j=2}^{r-1}2^{\frac{r-j}{2}}-1<0.\] 		
            Also, 
			\[P(2)= 2^{r+1}-2\sum_{j=1}^{r-1}2^{r-j}-1\]
			\[=2^{r+1}-2(2^{r-1}+2^{r-2}+\dots +2^{2}+2+1)+1\]
			\[=2^{r+1}-2(2^{r}-1)+1=3.\]
			That means, $P(2)>0.$
			
			Then, by Bolzano's theorem, there is at least one zero of the polynomial in the interval $[\sqrt{2},2]$, denote it as $\rho.$  Since, there is no integer in  $(\sqrt{2},2)$, $\rho$ must be an irrational number by \lemref{4.2}. This is the Cauchy bound for $P(x)$. From the earlier discussion and \thmref{perron}, there are no repeated roots in $P(x)$. 
            
            Let \(\rho_1,\rho_2,\dots,\rho_{r+1}\) be the roots of \(P(x)=0\). Since \(P\) has no repeated roots, the general solution of the recurrence is
\begin{equation}\label{rho}
|\mathcal{N}_{2l}|=c_1\rho_1^l+c_2\rho_2^l+\cdots+c_{r+1}\rho_{r+1}^l,
\end{equation}
where \(c_1,\dots,c_{r+1}\) are constants determined by the initial conditions. Let the real positive root $\rho$ be $\rho_k$. For simplicity, we denote the root $\rho_k$ by $\rho$ and the coefficient $c_k$ of $\rho^l_k$ by $C_1$.  

{Since \(|\rho_i|<\rho\) for all \(i\neq k\), we have
\[
\frac{|\mathcal{N}_{2l}|}{\rho^l}
=
C_1+\sum_{i\neq k} c_i\left(\frac{\rho_i}{\rho}\right)^l.
\]
The second term tends to \(0\) as \(l\to\infty\)}. 
the asymptotic follows once \lemref{4.9} proves $C_1>0$. \end{proof}

To extract the exact coefficient $C_1$, we rely on the standard rational expansion theorem for distinct roots. We recall the precise theorem given in \cite{GKP}.
\begin{theorem}[Rational Expansion Theorem for distinct roots]\label{co}
If
$F(z)=\frac{B(z)}{R(z)},$
where
$R(z)=q_0\prod_{k=1}^{s}(1-\rho_k z),$
 and the numbers \(\rho_1,\dots,\rho_s\) are distinct, and \(B(z)\) is a polynomial of degree less than \(s\), then
\[
[z^n]F(z)=\sum_{k=1}^{s} c_k \rho_k^n,
\qquad\text{where}~~c_k=-\frac{\rho_k\,B(1/\rho_k)}{R'(1/\rho_k)}.
\]
\end{theorem}
\begin{lemma}\label{4.9}
The coefficient of $\rho^l$, i.e., $C_1$ is a positive real number.
\end{lemma}

\begin{proof}
For simplification, let us consider $|\mathcal{N}_{2k}|$ as $a_k$.

So, we rewrite \eqnref{2.71} in terms of $a_k,$
\[
a_k = 2\sum_{j=1}^{r-1} a_{k-j-1} + a_{k-r-1}.
\]

This implies
\begin{equation}\label{new}
    a_k = 0 \cdot a_{k-1} + 2a_{k-2} + 2a_{k-3} + \cdots + 2a_{k-r} + a_{k-r-1}.
\end{equation}

Let us consider the power series
\[
f(z) = \sum_{k \geq 1} a_k z^{k-1}.
\]

We know the characteristic polynomial of the recurrence relation, 
\[
P(z) = z^{r+1} - 2z^{r-1} - \cdots - 2z^2 - 2z - 1.
\]
Then the reciprocal of $P$ as follows,
\begin{equation}\label{P}
    \tilde{P}(z) =z^{r+1} P\left(\frac{1}{z}\right)= 1 - 0\cdot z - 2z^2 - \cdots - 2z^{r-1} - 2z^r - z^{r+1}.
\end{equation}

Now, we consider another polynomial $N(z)$ such that
\[
\tilde{P}(z) f(z) = N(z).
\]

Then, we have
\[
N(z) = a_1 + (a_2 - 0\cdot a_1)z + (a_3 - 0\cdot a_2 - 2a_1)z^2 + \cdots
\]
\[
+ (a_k - 0\cdot a_{k-1} - 2a_{k-2} - \cdots - 2a_1)z^{k-1} + \cdots
\]
\[
+ (a_{n+1} - 0\cdot a_n - 2a_{n-1} - \cdots - 2a_{n-r+1} - a_{n-r})z^n + \cdots.
\]

 The coefficients of $z^n$ for $n \geq r+1$ are vanished, since we have the relation from \eqref{new},
\[
a_{n+1} - \left(0\cdot a_n + 2a_{n-1} + \cdots + 2a_{n-r+1} + a_{n-r}\right) = 0.
\]
Therefore, 
\[
N(z) = b_0 + b_1 z + b_2 z^2 + \cdots + b_r z^r,
\]
where
\[
b_0 = a_1,
\]
\[
b_1 = a_{2} - 0\cdot a_1,
\]
\[
b_2 = a_{3} - 0\cdot a_2 - 2a_{1},
\]
\[\vdots\]
\[
b_r = a_{r+1} - 0\cdot a_r - 2a_{r-1} - \cdots - 2a_1.
\]

\medskip

\noindent
Now, we have $f(z) = \dfrac{N(z)}{\widetilde P(z)},$ where the degree of $N(z)$ less than the degree of $\widetilde P(z)$. In our situation, we apply \thmref{co} with
\(
F(z)=f(z), ~ B(z)=N(z), ~ R(z)=\widetilde{P}(z).
\) Then, we have
\[
[z^{n}] f(z) =a_{n+1}= \sum_{k=1}^{r+1} c_k \rho_k^{n+1},
\qquad
\text{where}~~
c_k = - \frac{\rho_k \, N\!\left(\frac{1}{\rho_k}\right)}{\widetilde P'\!\left(\frac{1}{\rho_k}\right)}.
\]

{{If \(\rho_k=\rho\), then the coefficient corresponding to the term \(\rho^{n+1}\) is \(c_k:=C_1\)}.
To show $C_1$ is positive, we first show that 
 $\widetilde P'\!\left(\frac{1}{\rho}\right)<0$.}
From \eqnref{P}, we have 
\[  \widetilde{P}'(z) = - 2\cdot2z - \cdots - 2(r-1)z^{r-2} - 2rz^{(r-1)} - (r+1)z^{r}.\]

Since all coefficients of the terms $z^k$ for $k \geq 1$ in $\widetilde{P}(z)$ are non-positive and strictly negative for $k \geq 2$, it follows that $\widetilde{P}'(z) < 0$ for all $z > 0$. This implies $\widetilde P'\!\left(\frac{1}{\rho}\right) < 0$.
 
 Now, our claim is that every $b_k$ in $N(z)$ is non-negative.
To prove that, we use the initial conditions on $a_l$ from \lemref{count}.
Then
\[
b_0 = a_1 = |\mathcal{N}_{2}| = \frac{1}{6}(2 - 2) = 0,
\]

\[
b_1 = a_2 - 0 \cdot a_1 = a_2 = \frac{1}{6}(2^2 + 2),
\]

\[
b_2 = a_3 - 0 \cdot a_2 - 2a_1
 = \frac{1}{6}\left(2^3 - 2 - (2^2 + 2) \right),
\]

\[
= \frac{1}{6}\left(2^3 - 2 - 2^2 - 2 \right) > 0.\]

In general,
if $m$ is even and $m \neq r+1$, then
\[
b_m = \frac{1}{6}\left(2^m + 2 - 2\big(2^{m-2}+2 + 2^{m-3}-2 + \cdots + 2^2 +2+2-2\big)\right).
\]

Summing the geometric series,
\[
b_m = \frac{1}{6}\left(2^m + 2 - 2^2(2^{m-2} - 1)\right)>0.
\]

Similarly, if $m$ is odd and $m \neq r+1$, then
\[
b_m = \frac{1}{6}\left(2^m - 2 - 2\big(2^{m-2} + 2^{m-3} + \cdots + 2^2 + 2\big)\right)>0.
\]

In the case $m = r+1$, an extra term
\[
\frac{1}{6}\left(2^{u+1} - 2(-1)^{u+1}\right)
\]
appears, that is non negative. That means, $b_{r+1} > 0$. 
Therefore, in $N(z)$ there are no negative coefficients. In fact, except for the constant term, every coefficient is positive.

Thus, it follows that
\(
N(z) > 0 \quad \text{for all } z > 0.
\)
In particular,
\(
N\!\left(\frac{1}{\rho}\right) > 0.
\)

Hence,
\[
C_1 = - \frac{\rho \, N\!\left(\frac{1}{\rho}\right)}{\widetilde P'\!\left(\frac{1}{\rho}\right)} > 0.
\]
 
\end{proof}
    
		%\begin{lemma}\label{2.8}
		%	$R^{np}_{2l}$ is the set of the non-primitive reciprocal classes, then $|R^{np}_{2l}|=O \gamma_{1}c_{k}l\rho^{l/2}$
		%\end{lemma}
		\subsection{When $r$ is even }
	 Assume that $r=2u$ is even. By Lemma~\ref{count}, reciprocal conjugacy classes of odd word length occur, and these arise precisely from the symmetric $p$-reciprocal classes. Therefore, in this subsection we restrict to this setting.
		\begin{lemma}
			Suppose $ \mathcal{N}_{2l-1}$ be the set of conjugacy classes of reciprocal words of word length $2l-1$. Then we have :
			
			\begin{enumerate}
				\item If $l\leq r+u+1$, then, $|\mathcal{N}_{2l-1}|=\frac{1}{6}(2^{l-u-1}+2(-1)^{l-u-1})$.
				\item If $l= r+u+2$, then $|\mathcal{N}_{2l-1}|=\frac{1}{6}(2^{l-u-1}+2(-1)^{l-u-1})+\frac{ 1}{6}(2^{u+1}+2(-1)^{u+1})$.
				\item If $l\geq r+u+3$, 
				\[|\mathcal{N}_{2l-1}|=|\mathcal{N}_{2(l-u-1)}|=2\sum_{j=1}^{r-1}|\mathcal{N}_{2(l-u-j-2)}|+|\mathcal{N}_{2(l-u-r-2)}|.\]
			\end{enumerate}
		\end{lemma}
		\begin{proof}
			To prove $(1)$, it suffices to show that 
\[
|\mathcal{N}_{2l-1}|=|\mathcal{N}_{2(l-u-1)}|.
\]
Consider a reciprocal reduced word of length \(2l-1\) of the form
\[
\big(\io \gt \io \g^{k_{1}}\io \g^{k_{2}}\dots \io \g^{k_{n}}\io \g^{-k_{n}}\dots \io \g^{-k_{2}}\io \g^{-k_{1}}\big)^{\epsilon},
\]
where \(\epsilon\) is an odd positive integer. 

 {Since the base word is symmetric \(p\)-reciprocal, its length is
\[
L=2\sum_{i=1}^n |k_i|+2n+(r+1),
\]
and therefore \(L\ge r+1\). If \(\epsilon>1\) is odd, then \(\epsilon\ge 3\), so
\[
\epsilon L \ge 3L \ge 3(r+1).
\]
But the given word has length \(2l-1\le 3r\), a contradiction. Hence necessarily  $\epsilon=1$.}

Then, we get this equation for the word length $2l-1$ after rewriting the $k_{i}$'s,
			\[ 2\sum_{i=1}^{n}|k_{i}| +2n+2qr+2u+1= 2l-1\]
			\[\imp\sum_{i=1}^{n}|k_{i}| +n+qr= l-1-u. \]
            
			$ (1)$ follows from the above equality.

			$ (2) $ follows from the previous case and from the arguments in the proof of \lemref{count} and the fact that $ [(\io\gt)^{3}] $ also belongs to $ \mathcal{N}_{3(r+1)}. $
			
			For $ (3) $, this will follow from the proof of \lemref{count}. \end{proof}
		
		\medskip

		\begin{lemma}\label{5} $|\mathcal{N}_{2l-1}|\sim C_2\rho^{l}$ as $l\rightarrow \infty$ where, $ \rho $ is the real positive root of the equation $P(z)=0$ and $C_2=C_1\rho^{-u-1}$. 
		\end{lemma}
		\begin{proof}
			By the above arguments and  \lemref{2.7}, as $l\rightarrow \infty$ we see that $$|\mathcal{N}_{2l-1}|\sim C_1\rho^{l-u-1}\sim C_2\rho^{l},$$ where $C_2=C_1\rho^{-(u+1)}$.
		\end{proof}
		
		\subsubsection{Proof of \thmref{mthm}} 
		The theorem follows from \lemref{2.7} and \lemref{5}. \qed
		\medskip

Recall the Big $O$-notation, defined in \secref{not}. To establish \corref{cor1}, we observe the following lemmas. 
	\begin{lemma}
		Let $\beta = \lfloor \tfrac{l}{2} \rfloor$. Then,
		\[
		|\mathcal{N}^{np}_{l}| = O\left(\rho^{\left\lfloor \tfrac{\beta+1}{2} \right\rfloor} \right),
		\]
		where $\mathcal{N}^{np}_{l}$ denotes the set of all non-primitive reciprocal classes in $\Gamma_{2r}$ with word length $l$.
	\end{lemma}
	
	\begin{proof}
		Let $\mathcal{N}^{p}_{l}$ denote the set of all primitive reciprocal classes in $\Gamma_{2r}$ with word length $l$. Clearly, both $\mathcal{N}^{np}_{l}$ and $\mathcal{N}^{p}_{l}$ are sub collections of $\mathcal{N}_{l}$. Note that a non-primitive reciprocal class of word length $l$ arises from a primitive class of word length $q$ dividing $l$, where $q < l$. Therefore,
		\[
		|\mathcal{N}^{np}_{l}| \leq \sum_{q \mid l,~q < l} |\mathcal{N}^{p}_{q}| \leq \sum_{q = 1}^{\lfloor \frac{l}{2} \rfloor} |\mathcal{N}_{q}|.
		\]
		Using the fact that $|\mathcal{N}_q|\sim C\rho^{\left\lfloor \tfrac{q+1}{2} \right\rfloor}$ (where $C$ is either $C_1$ or $C_2$, depending on $q$), implies $\lvert\mathcal{N}_q\rvert=O\left(\rho^{\left\lfloor \tfrac{q+1}{2} \right\rfloor}\right)$, we obtain
		\[
		\sum_{q=1}^{\lfloor \frac{l}{2} \rfloor} |\mathcal{N}_{q}| = \sum_{q=1}^{\lfloor \frac{l}{2} \rfloor} O\left( \rho^{\left\lfloor \tfrac{q+1}{2} \right\rfloor} \right).
		\]
		Since the dominant term in this sum corresponds to the largest value of $q$, which is $ \lfloor \tfrac{l}{2} \rfloor = \beta$, we conclude that
		\[
		|\mathcal{N}^{np}_{l}| = O\left(\rho^{\left\lfloor \tfrac{\beta+1}{2} \right\rfloor} \right).
		\]
		This proves the lemma. 
	\end{proof}

		% number of reciprocal classes (whether it is odd or even) is asymptotic to an exponential function. And the primitive and non-primitive results are also similar but not equivalent.\\
		%	Note that here, the reciprocal words have with even and odd word length both. So, the counting of non-primitive reciprocal elements with even word length is not same result as \lemref{2.8} since, the non-primitive elements of word length $2l$ can be seen as power of an primitive element of word length even. 
		%\begin{lemma}
		%	$R^{np}_{2l-1}$ be the collection of reciprocal classes with word length $2l-1$. Then, $R^{np}_{2l-1}\sim \gamma_{1}c_{k}(l-u-1)\rho^{l/2}$.
		
		%\end{lemma}
		%\begin{proof}
		%	The proof is similar to \lemref{2.8}.
		%\end{proof}	\section{new}
	\begin{lemma}\label{prim}
			Let $\mathcal{N}_{l}$ be the set of all reciprocal classes in $\Gamma_{2r}$ with word length $l$, and $\mathcal{N}^{p}_{l}$ be the subset of primitive reciprocal classes. Then \[
		|\mathcal{N}^{p}_{l}| \sim |\mathcal{N}_{l}|.
		\]
		
	\end{lemma}
	
	\begin{proof}
		Let $\rho$ be the unique positive root (mentioned in \thmref{4.4}) with maximum absolute value among the roots of the recurrence relation defined by the characteristic equation \[x^{r+1} - 2\sum_{j=1}^{r-1} x^{r-j} - 1 = 0.\] By \thmref{mthm}, we have positive integers $M'$ and $l_0$, such that the total number of reciprocal classes satisfies
		\[
		|\mathcal{N}_{l}| \geq M'  \rho^{\left\lfloor \tfrac{l+1}{2} \right\rfloor}, \qquad \text{
        for all $l\geq l_0$.
        }
		\]
		and for positive integers $M$ and $l_1$ the number of non-primitive classes satisfies the upper bound
		\[
		|\mathcal{N}^{np}_{l}| \leq M \rho^{\left\lfloor \tfrac{\beta+1}{2} \right\rfloor} 
		\quad  \text{for all $l\geq l_1$},
		\]
	where $\text{for } \beta = \left\lfloor \tfrac{l}{2} \right\rfloor.$
		Also, by definition,
		\[
		|\mathcal{N}_{l}| = |\mathcal{N}^{np}_{l}| + |\mathcal{N}^{p}_{l}| \leq M \rho^{\left\lfloor \tfrac{\beta+1}{2} \right\rfloor} + |\mathcal{N}^{p}_{l}|.
		\]
		We obtain the inequality
		\[
		|\mathcal{N}_{l}| - M \rho^{\left\lfloor \tfrac{\beta+1}{2} \right\rfloor} \leq |\mathcal{N}^{p}_{l}| \leq |\mathcal{N}_{l}|.
		\]
		Dividing by $|\mathcal{N}_{l}|$, we get
		\[
		1 - \frac{M \rho^{\left\lfloor \tfrac{\beta+1}{2} \right\rfloor}}{|\mathcal{N}_{l}|}
		\leq \frac{|\mathcal{N}^{p}_{l}|}{|\mathcal{N}_{l}|} \leq 1
        \]
		\[
        1 - \frac{M \rho^{\left\lfloor \tfrac{\beta+1}{2} \right\rfloor}}{M'\rho^{\floor{\frac{l+1}{2}}}}
		\leq \frac{|\mathcal{N}^{p}_{l}|}{|\mathcal{N}_{l}|} \leq 1.
        \]
		Since \( \beta = \lfloor \tfrac{l}{2} \rfloor <l\), the denominator dominates as \( l \to \infty \), and the ratio tends to zero. Thus,
		\[
		\frac{|\mathcal{N}^{p}_{l}|}{|\mathcal{N}_{l}|} \to 1 \quad \text{as } l \to \infty,
		\]
		which implies
		\[
		|\mathcal{N}^{p}_{l}| \sim |\mathcal{N}_{l}|.
		\]
    This proves the lemma. 
	\end{proof}

\subsection{Proof of \corref{cor1}}
	The corollary follows from \lemref{prim} and \thmref{mthm}. \qed

			\subsection*{Acknowledgment} 
The authors sincerely thank the referees for their constructive comments that greatly improved the paper.

        \subsection*{Author Contributions}   
       All authors contributed equally to this work.
       
       \subsection*{Funding Declaration}
	 	Das acknowledges partial support from  the UGC-NFSC senior research fellowship for this work.   Gongopadhyay acknowledges the research grants\\ CRG/2022/003680 and
ANRF/ARGM/2025/000122/MTR.

	\end{document}